\documentclass{article}

\usepackage{arxiv}

\usepackage{amsmath}
\usepackage{amsthm}

\usepackage[utf8]{inputenc} % allow utf-8 input
\usepackage[T1]{fontenc}    % use 8-bit T1 fonts

\usepackage{hyperref}       % hyperlinks
\usepackage{url}            % simple URL typesetting
\usepackage{booktabs}       % professional-quality tables
\usepackage{amsfonts}       % blackboard math symbols
\usepackage{nicefrac}       % compact symbols for 1/2, etc.
\usepackage{microtype}      % microtypography
\usepackage{lipsum}		    % Can be removed after putting your text content
\usepackage{graphicx}
\usepackage{xcolor}
\usepackage{doi}
\usepackage{setspace}

\newtheorem{theorem}{Theorem}[section]
\newtheorem{lemma}{Lemma}[section]
\newtheorem{corollary}{Corollary}[theorem]

\newcommand{\dsp}{\displaystyle}
\newcommand{\PP}{\mathcal{P}}
\newcommand{\R}{\mathbb{R}}

\numberwithin{equation}{section}

\title{Complementary Romanovski-Routh polynomials and their zeros}

\author{ 
\href{https://orcid.org/0000-0002-1318-3892}{\includegraphics[scale=0.06]{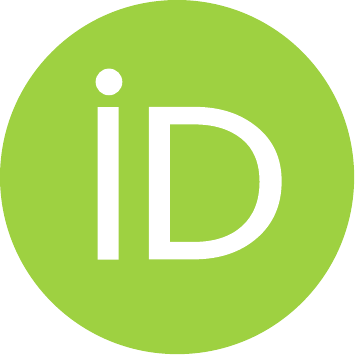}\hspace{1mm}Luana~L.~Silva~Ribeiro}\\
	Departamento de Matemática \\
	IBILCE - Universidade Estadual Paulista  \\
	S\~ao Jos\'e do Rio Preto, SP, CEP 15054-000\\
	\texttt{luana.math@hotmail.com} \\
\And
\href{https://orcid.org/0000-0002-5124-8423}{\includegraphics[scale=0.06]{orcid.pdf}\hspace{1mm}Alagacone~Sri~Ranga} \\
	Departamento de Matemática \\
	IBILCE - Universidade Estadual Paulista  \\
	S\~ao Jos\'e do Rio Preto, SP, CEP 15054-000\\
	\texttt{sri.ranga@unesp.br} \\
\And
\href{https://orcid.org/0000-0002-5918-7549}{\includegraphics[scale=0.06]{orcid.pdf}\hspace{1mm}Yen~Chi~Lun} \\
	Departamento de Matemática \\
	IBILCE - Universidade Estadual Paulista  \\
	S\~ao Jos\'e do Rio Preto, SP, CEP 15054-000\\
	\texttt{yen.chilun@yahoo.com.tw} \\		
}

\renewcommand{\shorttitle}{Complementary Romanovski-Routh polynomials and their zeros}

\hypersetup{
pdftitle={A template for the arxiv style},
pdfsubject={q-bio.NC, q-bio.QM},
pdfauthor={David S.~Hippocampus, Elias D.~Striatum},
pdfkeywords={First keyword, Second keyword, More},
}

\begin{document}
\maketitle

\begin{abstract}
The efficacy of numerical methods like integral estimates via Gaussian quadrature formulas depends on the localization of the zeros of the associated family of orthogonal polynomials. Following the renewed interest in quadrature formulas on the unit circle and $R_{II}$-type polynomials, in this work we present properties satisfied by the zeros of the complementary Romanovski-Routh polynomials. Our results include extreme bounds, convexity, density, and the connection of such zeros with the zeros of classical orthogonal polynomials via asymptotic formulas.
\end{abstract}

% keywords can be removed
\keywords{Asymptotics \and Complementary Romanovski-Routh polynomials \and Convexity \and Density \and Extreme bounds}

\onehalfspacing
\section{Introduction}\label{intro}
The $n$-th degree monic {\em complementary Romanovski-Routh polynomial} (CRR polynomial for short) that we consider here was presented in Mart\'inez-Finkelshtein et al. \cite{MRRT-CRR-2019} in terms of Gauss hypergeometric function as
$$
\widehat{\PP}_{n}(b;x) = (x-i)^{n} {}_{2}F_{1}\left(-n,b;b+\overline{b};\frac{-2i}{x-i}\right), \quad n \geq 0,
$$
where $b=\lambda+i\eta$, with $\lambda>0$ and $\eta\in \R$.
The CRR polynomials appear in wave functions in the context of quantum chromodynamics (see Raposo et al. \cite{RapWebCastKirc-CEJP2007}), and they also play an important role in the solution of the (one-dimensional) Schr\"{o}dinger equation for the hyperbolic Scarf potential (see Gautschi \cite{Gautschi-SIAMRev1967}). We also cite Shukla and Swaminathan \cite{Swaminathan2023} for a recent generalization of these polynomials. It is noteworthy that the sequence of CRR polynomials is also an example of an Appel sequence. This property was explored recently in Mart\'inez-Finkelshtein et al. \cite{MRRT-CRR-2019} to obtain a generating function for these polynomials. The latter belongs to a subfamily of the Whittaker functions that includes the Coulomb wave functions and the Bessel functions as particular cases.

From the point of view of  numerical analysis,  in the recent work Bracciali et al.  \cite{BraccPereiRanga-NA2019} it was shown that quadrature formulas based on these polynomials also provide useful  quadrature formulas on the unit circle.  Thus the study of the behavior and location of the zeros of the CRR polynomial is important.

We note that bounds for these zeros can be obtained from at least two distinct methods from the works of Jordaan and To\'okos \cite{Kerstin2014} and Mart\'inez-Finkelshtein et al. \cite{Veronese}. In fact, in Jordaan and To\'okos \cite{Kerstin2014} the authors have investigated properties of the Pseudo-Jacobi polynomials, which can be adapted to the present case by exploring the connection between the CRR polynomials and the complexified Jacobi polynomials (Raposo et al.  \cite{RapWebCastKirc-CEJP2007}). Alternatively, from the fact that CRR polynomials are related to special para-orthogonal polynomials (see Mart\'inez-Finkelshtein et al. \cite{MRRT-CRR-2019} for further details), the results found in Mart\'inez-Finkelshtein et al. \cite{Veronese} can also be used to construct bounds for the zeros of CRR polynomials. However, the bounds obtained by these methods are not always optimal. In this work we shall present new bounds which are more accurate for large variations in the parameters $\lambda$ and $\eta$. Furthermore, we also investigate properties such as  convexity, density, and asymptotics. From the latter we also show how the CRR polynomial is linked to the classical Hermite and Laguerre polynomials.

In what follows, we adopt the more convenient normalization for the CRR polynomial $\PP_{n}(b;x)$ (see Mart\'inez-Finkelshtein et al. \cite{MRRT-CRR-2019}) given by
$$
\PP_{n}(b;x)  =  \frac{(2\lambda)_{n}}{2^{n} (\lambda)_{n}} \widehat{\PP}_{n}(b;x), \quad n \geq 0,
$$
where $(\alpha)_{n}=\alpha(\alpha+1)\cdots(\alpha+n-1)$ denotes the Pochhammer symbol, and $b=\lambda+i\eta$, $\lambda>0$. In this work we study the zeros of $\PP_{n}(b;x)$, henceforth denoted by $x_{n,k}^{(\eta,\lambda)}$, for $k =1,2, \ldots, n$.

The paper is organized as follows. We leave to Section \ref{prelim} the presentation of all the necessary notation, tools, and theorems which are going to be used throughout the work. We also recall in this section how the CRR polynomials are constructed via the $R_{II}$-type recurrence relation, and that their zeros coincide with the solutions of certain generalized eigenvalue problems. Sections \ref{main} and \ref{main2} summarize the novel results of our work. Section \ref{main} presents the extreme bounds, convexity, and density of the zeros of CRR polynomials. Moreover, these new bounds are compared with the two others obtained by the methods mentioned before. We simulate and verify that for the majority of $\eta$ and $\lambda$ the bounds presented here are more accurate. We also state some orthogonality property of the $k$-th associated CRR polynomials. Finally,
Section \ref{main2} is devoted to the asymptotic behavior of the CRR polynomials and their zeros.

\section{Preliminaries}\label{prelim}
The sequence $\{\PP_{n}(b;x)\}_{n \geq 0}$ of CRR polynomials is solution of the $R_{II}$-type recurrence relation (see Ismail and Ranga \cite{Ismail-Ranga})
\begin{equation}\label{3termR2TTT}
\PP_{n+1}(b;x)=(x-c_{n+1}^{(b)})\PP_{n}(b;x)-d_{n+1}^{(\lambda)}(x^{2}+1)\PP_{n-1}(b;x), \quad  n\geq 1,
\end{equation}
where  $\PP_{0}(b;x)=1$,  $\PP_{1}(b;x)=x-c_{1}^{(b)}$,
\begin{equation}\label{RomonovskyCoeff}
c_{n}^{(b)} =\frac{\eta}{\lambda+n-1} \ \ \mbox{and} \ \ d_{n+1}^{(\lambda)}=\frac{1}{4}\frac{n(n+2\lambda-1)}{(n+\lambda-1)(n+\lambda)},  \ \  n\geq 1.
\end{equation}
The sequence $\{d_{n}^{(\lambda)}\}_{n \geq 2}$ is a positive chain sequence, whose detailed properties can be found for instance in Wall \cite{Wall-book1948}, and in Chihara \cite{Chihara-book1978}. Thus it follows (see Theorem 2.2 of Ismail and Ranga \cite{Ismail-Ranga}) that the zeros of $\PP_{n}(b;x)$ are real, simple, and can be denoted as %
\[
x_{n,1}^{(\eta,\lambda)}< x_{n,2}^{(\eta,\lambda)}<\ldots< x_{n,n-1}^{(\eta,\lambda)} < x_{n,n}^{(\eta,\lambda)}, \ \mbox{ for} \ n \geq 2.
\]
It is also known that the zeros of $\PP_{n-1}(b;x)$ and $\PP_{n}(b;x)$ interlace.

Furthermore, $\PP_{n}(b;x)$ is the characteristic polynomial of the generalized eigenvalue problem %
\begin{equation}\label{GenEigeProb}
  \textbf{A}_{n}^{(b)}\textbf{u}_{n}^{(b)}(x)=x\textbf{B}_{n}^{(b)}\textbf{u}_{n}^{(b)}(x),
\end{equation}
where the Hermitian tridiagonal matrices $\textbf{A}_{n}^{(b)}$ and $\textbf{B}_{n}^{(b)}$ are
$$
\begin{bmatrix}
c_{1}^{(b)} & i\sqrt{d_{2}^{(\lambda)}} & \ldots & 0 & 0 \\
   -i\sqrt{d_{2}^{(\lambda)}} & c_{2}^{(b)} & \ldots & 0 & 0 \\
   \vdots  & \vdots & \ddots & \vdots & \vdots\\
   0 & 0  & \ldots & c_{n-1} & i\sqrt{d_{n}^{(\lambda)}}\\
   0 & 0  & \ldots & -i\sqrt{d_{n}^{(\lambda)}} & c_{n}^{(b)}
\end{bmatrix}\ \ \ \textup{and} \ \ \
\begin{bmatrix}
1 & \sqrt{d_{2}^{(\lambda)}}  & \ldots & 0 & 0 \\
\sqrt{d_{2}^{(\lambda)}} & 1  & \ldots & 0 & 0 \\
\vdots & \vdots & \ddots & \vdots & \vdots \\
0 & 0  & \ldots & 1 & \sqrt{d_{n}^{(\lambda)}}\\
0 & 0  & \ldots & \sqrt{d_{n}^{(\lambda)}} & 1
\end{bmatrix},
$$
respectively. Polynomial sequences generated by $R_{II}$-type recurrence relations were initially investigated by Ismail and Mason \cite{Masson}, and their link to generalized eigenvalue problems was first explored by Zhedanov \cite{Zhedanov-JAT1999}. Another property presented by the $n$-th degree CRR polynomial used throughout the work is that it satisfies the second order differential equation 
\begin{equation}\label{RomonoviskyEDO}
A(x)\PP^{\prime\prime}_{n}(b;x)-2B(n,\lambda,\eta;x)\PP^{\prime}_{n}(b;x)+C(n,\lambda)\mathcal{P}_{n}(b;x)=0,
\end{equation}
with $A(x)=x^{2}+1$, $B(n,\lambda,\eta;x) =(n+\lambda-1)x-\eta$ and $C(n,\lambda) = n(n+2\lambda-1)$. This equation is used in conjunction with the following results.
\begin{theorem}[Sturm Comparison Theorem]\label{SturmTheoClassic}
Let
$y$ and $Y$ be nontrivial solutions of the differential equations
$$
\begin{array}{rl}
y''(x) + f(x) y(x)\!\!&=0,\\[1.5ex]
Y''(x) + F(x) Y(x)\!\!&=0,
\end{array}
$$
respectively, where $f, F \in C(r,s)$ are continuous functions in the interval $(r,s)$, and $f(x)\leq F(x)$, $f\not\equiv F$, in $(r,s)$. Let $x_1$ and $x_2$ with $r< x_1 < x_2 < s$ be two
consecutive zeros of $y$. Then the function $Y$ has at least one sign variation in the interval
$(x_1, x_2)$.
\end{theorem}
The Sturm Comparison Theorem can be found in Szeg\H{o}'s \cite{szego} or in Hille's book \cite{Hille}. We use it to obtain the Theorem \ref{densityR2Zeros}, and the Corollary \ref{Cor-densityR2Zeros} describing the density of the zeros of CRR polynomials and their $k$-associated CRR polynomials defined in  \eqref{3termR2-k}. We recall that the Sturm Comparison Theorem provides bounds for the zeros of several special functions (see Szego \cite{szego}). Moreover, it can be applied to prove the Convexity Theorem, which is convenient to study the convexity and spacing between consecutive zeros of polynomial functions. For further reading about convexity of some classical orthogonal polynomials we cite Jordaan and To\'okos \cite{Kerstin}. We use the Convexity Theorem with some appropriated variable transformation to demonstrate the Theorem  \ref{Thm-convexityR2Zeros}.

\begin{theorem}[Convexity Theorem]\label{Thm-convexity}
Let $y$ be a nontrivial solution of differential equation
$$
y''(x) + f(x) y(x)=0,
$$
where $f\in C(r,s)$ is a continuous function in the interval $(r,s)$ and $x_{1}<\ldots<x_{n}$ are the zeros of $y$, also in the interval $(r,s)$.
\begin{itemize}
 \item[\textup{\textbf{I}.}] If $f$ is decreasing in $(r,s)$, then $x_{k}-x_{k-1}<x_{k+1}-x_{k}$ for each $k=2,\ldots,n-1$;

 \item[\textup{\textbf{II}.}] If $f$ is increasing in $(r,s)$, then $x_{k}-x_{k-1}>x_{k+1}-x_{k}$ for each $k=2,\ldots,n-1$.
\end{itemize}
The zeros of $y$ are called convex (Resp. concave) on the interval $(r,s)$ if \textbf{I}  (Resp. \textbf{II}) holds.
\end{theorem}

Finally, in order to present upper and lower bounds for the largest and smallest zero of the CRR polynomial, we need the next technical lemma from Dimitrov and Genno \cite{DimitrovGeno}, which is a particular case of an inequality that holds for functions in the Laguerre-P\'olya class. Further information on the latter can be found, for instance, in Nikolov and Uluchevin \cite{Geno}.
\begin{lemma}\label{Lema-DG}
Let $f$ be a polynomial with only real zeros and degree $n\geq 4$. If $f(\zeta)=0$, then
$$
3(n-2)[f^{\prime\prime}(\zeta)]^{2}-4(n-1)f^{\prime}(\zeta)f^{\prime\prime\prime}(\zeta)\geq0.
$$
\end{lemma}
\section{Extreme Bounds, convexity and density}
\label{main}

We start this section by presenting the bounds for the largest and the smallest zero of the polynomial $\PP_{n}(b;x)$ in the next theorem, which is applied later on to obtain inequalities for the distance between two consecutive zeros of CRR polynomial in the Corollary \ref{coro-dist-zeros}.

\begin{theorem}\label{Thm-bounds}
  Let $n \geq  4$.  Then
  $$
  \frac{\eta  [n^2+2 \lambda +n(\lambda -1) ]- (n-1)\sqrt{ \Delta_{n}}}{\lambda  [2(n-1)+(n+2)\lambda]}<x_{n,1}^{(\eta,\lambda)}
<
  x_{n,n}^{(\eta,\lambda)}<\frac{\eta  [n^2+2 \lambda +n(\lambda -1)]+ (n-1)\sqrt{ \Delta_{n}}}{\lambda  [2(n-1)+(n+2)\lambda]},
  $$
  where
  $$
  \Delta_{n}= [\eta ^2+\lambda  (\lambda +2)]n^2+  2 \lambda (\eta ^2+\lambda ^2+2 \lambda -3)n+4 \lambda  [\eta ^2+(\lambda -1)^2].
  $$
\end{theorem}

\begin{proof}
We know that $\PP_{n}(b;x)$ satisfies Equation \eqref{RomonoviskyEDO}. Let $y(x) = \PP_{n}(b;x)$.  If $y(\zeta) = \PP_{n}(b;\zeta)=0$, then straightforward manipulations lead to

\begin{equation} \label{Eq-Dif-Eq-2}
y^{\prime\prime}(\zeta)=2\frac{(n+\lambda-1)\zeta-\eta}{\zeta^{2}+1}y^{\prime}(\zeta),
\end{equation}
and
$$
y^{\prime\prime\prime}(\zeta)=\frac{2[(n+\lambda-2)\zeta-\eta]y^{\prime\prime}(\zeta)-[(n-2)(n+\lambda-1)+n\lambda]y^{\prime}(\zeta)}{\zeta^{2}+1}.
$$

Thus, from Lemma \ref{Lema-DG}, we have
$$
\left[2\frac{y^{\prime}(\zeta)}{\zeta^{2}+1}\right]^{2}Q(\zeta)\geq0,
$$
where we defined $Q(x)=q_{2}x^{2}+q_{1}x+q_{0}$ with
\begin{align*}
q_{2}&=- \lambda  \left[(n+2)\lambda+ 2(n-1)\right],\\
q_{1}&=2 \eta  \left[n(n+\lambda-1)+2\lambda\right],\\
q_{0}&= n^3+2 (\lambda -2) n^2- \left(\eta ^2+4 \lambda -5\right)n -2 \left(\eta ^2-\lambda +1\right).
\end{align*}
The two real roots of the polynomial $Q(x)$ are
\[
x_{1,2}=\frac{\eta  [n^2+2 \lambda +n(\lambda -1) ]\pm(n-1)\sqrt{ \Delta_{n}}}{\lambda  [2(n-1)+(n+2)\lambda]},
\]
where $\Delta_{n}$ is defined in the theorem's statement, and we note that $\Delta_{n}>0$.  Since $q_{2}$ is negative, it follows that all the zeros of $\PP_{n}(b;x)$ are in the interval $(x_{1},x_{2})$.
\end{proof}

As anticipated, we solved numerically for the largest and smallest zeros of some representative CRR polynomials, and the corresponding bounds build from Theorem \ref{Thm-bounds}. Moreover, those bounds were tested and compared with previous bounds found in Jordaan and To\'okos \cite{Kerstin2014} and Mart\'inez-Finkelshtein et al. \cite{Veronese}, which can be determined with a few preliminary manipulations. The bounds found in Jordaan and To\'okos \cite{Kerstin2014} are connected to the CRR polynomials via their relation with the complexifield Jacobi polynomials, whereas the bounds given in Mart\'inez-Finkelshtein et al. \cite{Veronese} follow from the connection between $\PP_{n}(b;x)$ and the para-orthogonal polynomials explored in Section 2 of Mart\'inez-Finkelshtein et al. \cite{MRRT-CRR-2019}.

Our findings are summarized in Tables 1 and 2. The columns labeled by RRY contain the results from Theorem \ref{Thm-bounds}, whereas the ones labeled by JT and MRV contain the results found with aid of Jordaan and To\'okos \cite{Kerstin2014} and Mart\'inez-Finkelshtein et al. \cite{Veronese}, respectively. Moreover, for a given set of parameters, sharper bounds were highlighted, and the lower and upper bounds are discriminated by the subscripts $min$ and $max$, respectively. Our simulations for different values of $n$, $\lambda$ and $\eta$ reveal some interesting patterns. As shown in Table 1, by fixing $\eta=2$ and $n=30$, the bounds given by Theorem \ref{Thm-bounds} are less accurate than the previous ones when the parameter $\lambda$ is small. For example, for $\lambda=0.75$ and $\lambda=1.75$, the bounds obtained in Jordaan and To\'okos \cite{Kerstin2014} and Mart\'inez-Finkelshtein et al. \cite{Veronese} are sharper. However, as $\lambda$ increases the bounds from Theorem \ref{Thm-bounds} are always sharper, as shown in Table 1 for $\lambda$ as $5,10,15,20,25$ and $70$. Furthermore, even for smaller values of $\lambda$, Theorem \ref{Thm-bounds} seems to always provide better bounds if $\eta$ is not too small, as shown in Table 2.

\begin{center}
\begin{table}[!ht]%
\centering
\caption{Smallest and largest zeros, and extreme bounds when $\eta=2$ and $n=30$. The
$RRY_{max}$ and $RRY_{min}$ are the upper and lower extreme bounds, respectively, for the zeros obtained from Theorem \ref{Thm-bounds}. $JT_{max}$, $MRV_{max}$, $JT_{min}$ and $MRV_{min}$ are the upper and lower extreme bounds obtained by Jordaan and To\'okos \cite{Kerstin2014} and Mart\'inez-Finkelshtein et al. \cite{Veronese}, respectively.
\label{table3-bounds}}%
\begin{tabular*}{350pt}{@{\extracolsep\fill}crrrr@{\extracolsep\fill}}
\toprule
\textbf{$\lambda$} & \textbf{$x_{n,1}^{(\eta,\lambda)}$}  & \textbf{$RRY_{min}$}  & \textbf{$MRV_{min}$}  & \textbf{$JT_{min}$} \\
\midrule
0.75 & $-3.60366$ & $-6.23369$ &  $-6.40957$ &  $\bf{-4.97174}$ \\
1.75 & $-3.56614$ & $-5.49264$ &  $\bf{-5.12489}$ &  $-5.13691$ \\
5    & $-2.94731$ & $\bf{-3.83491}$ &  $-4.12943$ &  $-4.05698$ \\
10   & $-2.24406$ & $\bf{-2.69331}$ &  $-2.96588$ &  $-2.87276$ \\
15   & $-1.84833$ & $\bf{-2.14342}$ &  $-2.34765$ &  $-2.28402$ \\
20   & $-1.59757$ & $\bf{-1.81694}$ &  $-1.97708$ &  $-1.93278$ \\
25   & $-1.42313$ & $\bf{-1.59827}$ &  $-1.72983$ &  $-1.69740$ \\
70   & $-0.83027$ & $\bf{-0.89890}$ &  $-0.95410$ &  $-0.94747$ \\
\midrule
\textbf{$\lambda$} & \textbf{$x_{n,n}^{(\eta, \lambda)}$}  & \textbf{$RRY_{max}$}  & \textbf{$MRV_{max}$}  & \textbf{$JT_{max}$}\\
\midrule
0.75 & $51.93148 $ & $64.38003$ & $133.35633$ &  $\bf{62.30239}$\\
1.75 & $18.22559$ & $24.05906$ & $\bf{19.87006}$ &  $24.48929$ \\
5    & $5.99823$ & $\bf{7.61473}$ & $8.41346$ &  $8.05698$ \\
10   & $3.31253$ & $\bf{3.95257}$ & $4.44314$ &  $4.20609$  \\
15   & $2.43428$ & $\bf{2.81256}$ & $3.11345$ &  $2.98990$  \\
20   & $1.98455$ & $\bf{2.24960}$ & $2.46186$ &  $2.38732$  \\
25   & $1.70595$ & $\bf{1.90970}$ & $2.07325$ &  $2.02333$  \\
70   & $0.90430$ & $\bf{0.97623}$ & $1.03484$ &  $1.02684$  \\
\bottomrule
\end{tabular*}
%\begin{tablenotes}
%\item Source: Example for table source text.
%\item[$\dagger$] Example for a first table footnote.
%\item[$\ddagger$] Example for a second table footnote.
%\end{tablenotes}
\end{table}
\end{center}

\begin{center}
\begin{table}[!ht]%
\centering
\caption{Smallest and largest zeros, and extreme bounds when $\lambda=1.5$ and $n=4$. The
$RRY_{max}$ and $RRY_{min}$ are the upper and lower extreme bounds, respectively, for the zeros obtained from Theorem \ref{Thm-bounds}. $JT_{max}$, $MRV_{max}$, $JT_{min}$ and $MRV_{min}$ are the upper and lower extreme bounds obtained by Jordaan and To\'okos \cite{Kerstin2014} and Mart\'inez-Finkelshtein et al. \cite{Veronese}, respectively.\label{table3-bounds}}%
\begin{tabular*}{350pt}{@{\extracolsep\fill}crrrr@{\extracolsep\fill}}
\toprule
\textbf{$\eta$} & \textbf{$x_{n,1}^{(\eta, \lambda)}$}  & \textbf{$RRY_{min}$}  & \textbf{$MRV_{min}$}  & \textbf{$JT_{min}$} \\
\midrule
0   & $-1.18804$ & $-1.41421$      & $\bf{-1.28644}$ & $-1.61803$      \\
0.5 & $-0.81060$ & $-1.00000$      &  $\bf{-0.97680}$& $-1.13951$      \\
1.5 & $-0.31265$ & $-0.43303$      &  $-0.50471$     &$\bf{-0.41745}$  \\
5   & $-1.84833$ & $\bf{-2.14342}$ &  $-2.34765$ &  $-2.28402$  \\
7   & $0.96538$ & $\bf{0.91036}$ &  $0.76104$ &  $0.85254$  \\
15   & $2.33259$ & $\bf{2.25266}$ &  $1.98870$ &  $2.07881$ \\
35   & $5.58634$  & $\bf{5.41883}$  &  $4.83529$ &  $4.99376$  \\
90   & $14.43584$ & $\bf{14.01430}$ &  $12.52994$ &  $12.91264$   \\
\midrule
\textbf{$\eta$} & \textbf{$x_{n,n}^{(\eta,\lambda)}$}  & \textbf{$RRY_{max}$}  & \textbf{$MRV_{max}$}  & \textbf{$JT_{max}$}\\
\midrule
0   &  $1.18804$ & $1.41421$      & $\bf{1.28644}$ & $1.61803$  \\
0.5 &  $1.68062$ & $1.9333$       & $\bf{1.69224}$ & $2.18714$  \\
1.5 &  $2.94807$ & $3.23303$      & $\bf{3.31012}$ &  $3.56030$ \\
5   &  $2.43428$ & $\bf{2.81256}$ & $3.11345$ &  $2.98990$  \\
7   &  $11.54147$ & $\bf{12.15630}$ & $13.75984$ &  $12.96786$  \\
15   & $24.51550$ & $\bf{25.74734}$ & $29.32890$ &  $27.39236$  \\
35   & $57.08690$ & $\bf{59.91450}$ & $68.35048$ &  $63.70115$  \\
90   & $146.73775$ & $\bf{153.98571}$ & $175.71725$ &  $163.69722$  \\
\bottomrule
\end{tabular*}
%\begin{tablenotes}
%\item Source: Example for table source text.
%\item[$\dagger$] Example for a first table footnote.
%\item[$\ddagger$] Example for a second table footnote.
%\end{tablenotes}
\end{table}
\end{center}

%\vfill \pagebreak

Numerical simulations for different values of $\lambda$ also show that for larger values of $\lambda$ the results given by Theorem \ref{Thm-bounds} are somewhat optimal. In fact, in Figure \ref{fig1-1} we have plotted the zeros and the bounds as functions of the parameter $\lambda$ for two different and fixed sets of values $n$ and $\eta$. In Figure \ref{fig1-1} the continuous curves from bottom ($k=1$) to top ($k=n$) represent the plots of  $x_{n,k}^{(\eta,\lambda)}$ as a function of $\lambda$ with a fixed $n$ and $\eta$. The lower dashed curve represents the lower extreme bound for the zeros and the upper dot-dashed curve represents the upper extreme bound for the zeros, both obtained from Theorem \ref{Thm-bounds}. Clearly, these simulations suggest that for larger values of $\lambda$ the upper dot-dashed curve approaches the curve of $x_{n,n}^{(\eta,\lambda)}$ and similarly for the lower dashed curve with the one from $x_{n,1}^{(\eta,\lambda)}$.

\begin{figure*}[ht!]
\begin{minipage}[b]{0.48\linewidth}
 \centering
 \includegraphics[width=0.90\linewidth]{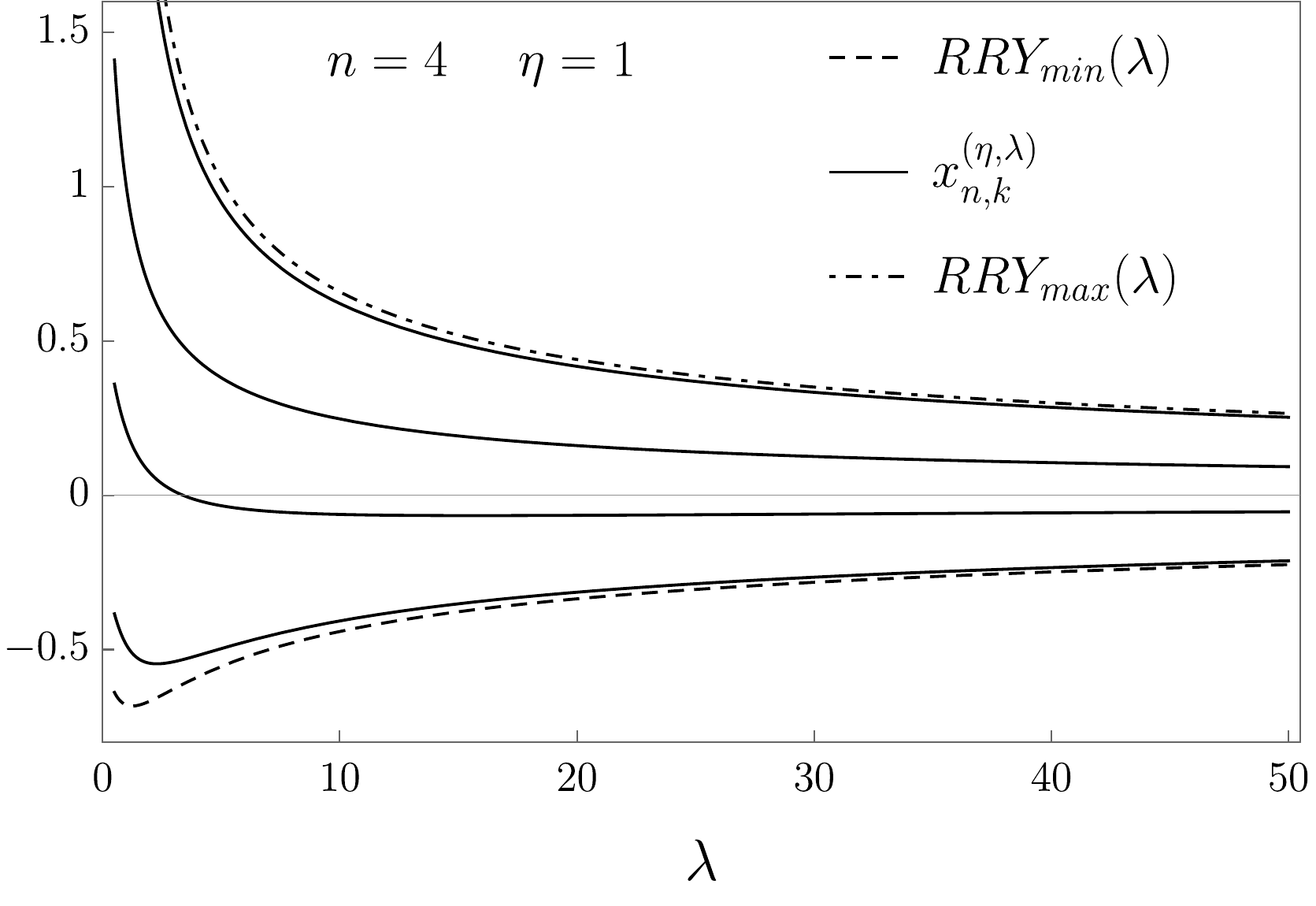} \\[-1ex]
% $(a)\  n = 4, \eta = 1$
\end{minipage} \hspace{0.02\linewidth}
\begin{minipage}[b]{0.48\linewidth}
 \centering
 \includegraphics[width=0.90\linewidth]{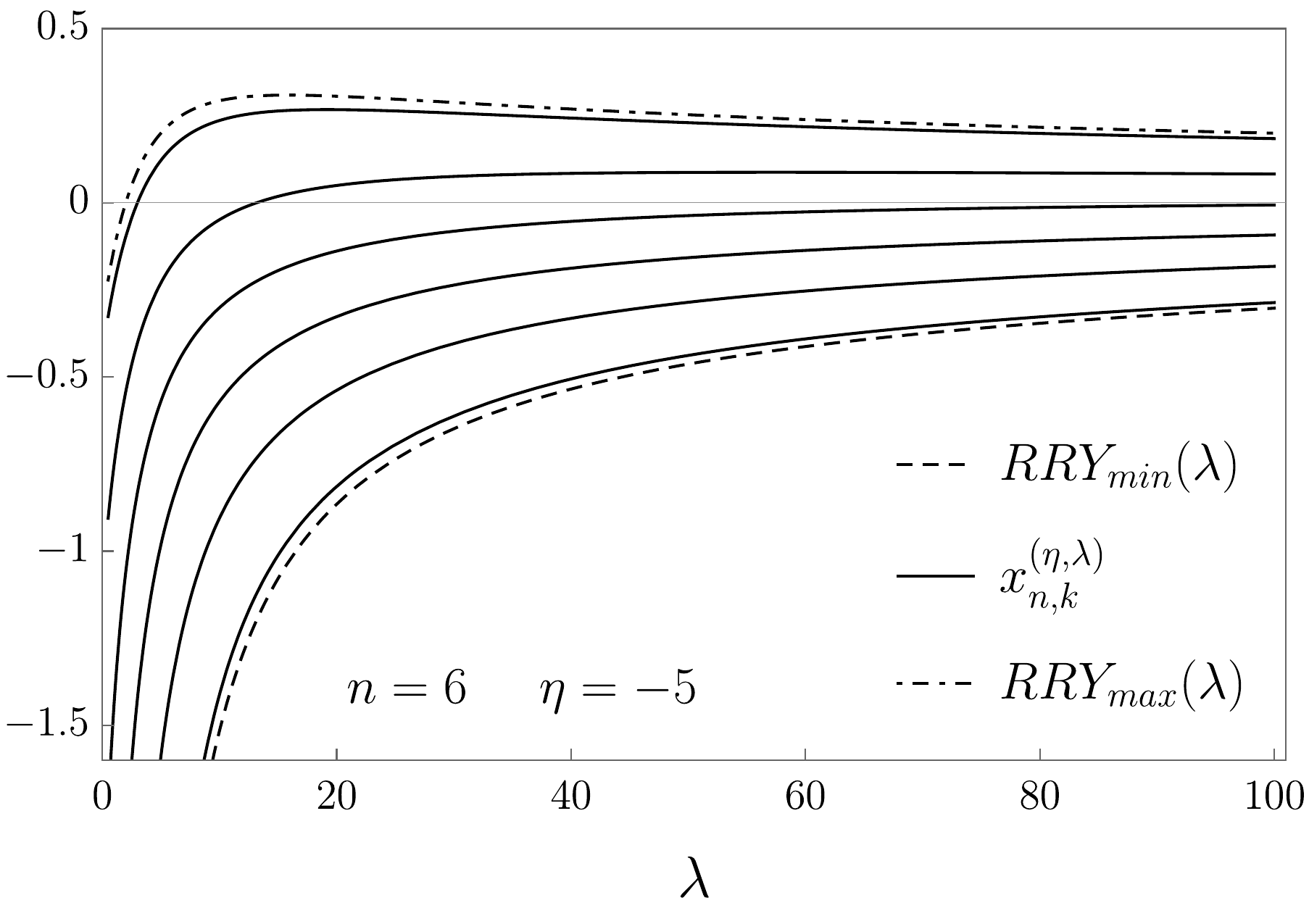} \\[-1ex]
% $(b)\ n = 6, \eta = -5$
\end{minipage}
\caption{ The inner curves represent the graphs of $x_{n,k}^{(\eta,\lambda)}$, $k=1, \ldots, n$,  as functions of $\lambda$ for fixed values of $n$ and $\eta$. The $RRY_{min}$ and $RRY_{max}$ are the upper and lower extreme bounds, respectively, for the zeros obtained from Theorem \ref{Thm-bounds}. }
\label{fig1-1}
\end{figure*}

\vfill \pagebreak

The following theorem gives some information about the convexity of the zeros of the $n$-th CRR polynomial with respect to the parameters $\lambda$ and $\eta$.

\begin{theorem}\label{Thm-convexityR2Zeros}
Let $\mathfrak{M}_{n}=\mathfrak{M}_{n}(b)=\eta(n+\lambda)/\lambda(\lambda-1)$ with $b=\lambda+i\eta$, $\lambda>0$ and $\eta\in\mathbb{R}$.
If $\lambda > 1$ and
\begin{enumerate}
\item[\textup{\textbf{I}.}] $\eta=0$, then the negative (Resp. positive) zeros of $\PP_{n}(b;x)$ are concave (Resp. convex);
\item[\textup{\textbf{II}.}] $\eta>0$, then the negative zeros of $\PP_{n}(b;x)$ are concave and the positive zeros of $\PP_{n}(b;x+\mathfrak{M}_{n})$ are convex;
\item[\textup{\textbf{III}.}] $\eta<0$, then the negative zeros of $\PP_{n}(b;x+\mathfrak{M}_{n})$ are concave and the positive zeros of $\PP_{n}(b;x)$ are convex.
\end{enumerate}
If $0<\lambda<1$ and
\begin{enumerate}
\item[\textup{\textbf{VI}.}] $\eta>0$, then the zeros of $\PP_{n}(b;x)$ in $(\mathfrak{M}_{n},0)$ are concave;
\item[\textup{\textbf{V}.}] $\eta<0$, then the zeros of $\PP_{n}(b;x)$ in $(0,\mathfrak{M}_{n})$ are convex;
\end{enumerate}
If $\lambda=1$ and $\eta>0$ (Resp. $\eta<0$), then all negative (Resp. positive) zeros are concave (Resp. convex).
\end{theorem}

\begin{proof}

The Differential Equation \eqref{RomonoviskyEDO} in the Sturm-Liouville form reads
\begin{equation}\label{RomonoviskyEDO S-L form with theta}
v_{n}^{\prime\prime}(\theta)+\Lambda_{n}(\theta) v_{n}(\theta)=0,
\end{equation}
where
\begin{align} \label{RomonoviskyEDO S-L form solution with theta}
\begin{split}
\Lambda_{n}(\theta)& =\lambda(1-\lambda)\tan^{2}(\theta)+2\eta(n+\lambda)\tan(\theta)+n^{2}+2n\lambda+\lambda-\eta^{2}, \\
v_{n}(\theta)& = e^{\eta\theta} \cos(\theta)^{n+\lambda}\,  \PP_{n}(b;\tan(\theta)),
\end{split}
\end{align}
with $\theta\in(-\pi/2,\pi/2)$. The details of the transformation to the Sturm-Liouville form can be found in Section 1.8 of Szeg\H{o} \cite{szego}. We note that the zeros $\theta_{n,k}^{(\eta,\lambda)}$ of $v_{n}(\theta)$ are such that $x_{n,k}^{(\eta,\lambda)} = \tan(\theta_{n,k}^{(\eta,\lambda)})$ and
$$
-\pi/2 < \theta_{n,1}^{(\eta,\lambda)} < \theta_{n,2}^{(\eta,\lambda)} < \cdots < \theta_{n,n}^{(\eta,\lambda)} < \pi/2.
$$
From Equation \eqref{RomonoviskyEDO S-L form solution with theta} it follows that
\begin{equation*}\label{derivative}
  \frac{d\Lambda_{n}(\theta)}{d\theta}=\frac{2}{\cos^{2}(\theta)}[\eta(n+\lambda)-\lambda(\lambda-1)\tan(\theta)]=0
\end{equation*}
if and only if $\tan(\theta)=\eta(n+\lambda)/\lambda(\lambda-1)$.

Suppose first that $\lambda>1$. Then if $\eta>0$, we have that $\mathfrak{M}_{n}>0$, $\Lambda_{n}$ is decreasing in the interval $(\arctan(\mathfrak{M}_{n}),\pi/2)$, and increasing in $(-\pi/2,\arctan(\mathfrak{M}_{n}))$. Therefore, by Theorem  \ref{Thm-convexity}, the negative zeros of $v_{n}(\theta)$ are concave and the positive zeros of $v_{n}(\theta)$ that belong to the interval $(\arctan(\mathfrak{M}_{n}),\pi/2)$ are convex. Since the tangent function is concave in $(-\pi/2,0)$, convex in $(\arctan(\mathfrak{M}_{n}),\pi/2)\subset(0,\pi/2)$ and increasing in $(-\pi/2,\pi/2)$, we conclude that the negative zeros of $\PP_{n}(b;x)$ are concave and the positive zeros of $\PP_{n}(b;x)$ which are in the interval $(\mathfrak{M}_{n},\infty)$ are convex. When $\eta<0$, the statement is a consequence of the identity $(-1)^{n}\PP_{n}(\overline{b};-x)=\PP_{n}(b;x)$, where $\overline{b}=\lambda-i\eta$, and for the case of $\eta=0$, the proof is immediate.

Finally, the proof when $0<\lambda<1$ is analogous, and for the remaining $\lambda=1$ case, we note that sign($d\Lambda_{n}(\theta)/d\theta$)=sign($\eta$), where $\textup{sign}(\eta)=1$ if $\eta\geq 0$ and $\textup{sign}(\eta)=-1$ if $\eta<0$, and the conclusion follows immediately by Theorem \ref{Thm-convexity}.
\end{proof}

It should be remarked that our result does not establish the convexity behavior for all the zeros for a given parameter choice. For instance, when $\lambda>1$ and $\eta>0$, the above theorem asserts that the zeros located on the right hand side of $\mathfrak{M}_{n}$ are convex and the ones located on the left hand side of the origin are concave. However, for this case, our result cannot be used to draw any conclusion about the zeros located in the interval $(0,\mathfrak{M}_{n})$.

The next corollary presents inequalities for the parametrized distance between two consecutive zeros of the $n$-th CRR polynomial.

\begin{corollary}\label{coro-dist-zeros} Let $-\pi/2 < \theta_{n,1}^{(\eta,\lambda)} < \theta_{n,2}^{(\eta,\lambda)} < \cdots < \theta_{n,n}^{(\eta,\lambda)} < \pi/2
$ be given by $\theta_{n,k}^{(\eta,\lambda)}=\textup{arctan}(x_{n,k}^{(\eta,\lambda)})$, for $k=1,\ldots,n$, and let
\begin{align*}
& f_{n}^{(\eta,\lambda)}=[n^2+\lambda(2n+1)]\left\{1-\frac{\eta^2}{\lambda(1-\lambda)}\right\}, \\
& g_{n}^{(\eta)}= (n+1)^2 + \frac{1}{3n}\left\{\eta^2[n(n+1)^2+n^3+4]+\textup{sign}(\eta)\,2\eta(n^2-1)\sqrt{\eta^2(n+1)^2+3(\eta^2+n^2)}\right\},
\end{align*}
where $\textup{sign}(\eta)=1$ if $\eta\geq 0$ and $\textup{sign}(\eta)=-1$ if $\eta<0$.
\begin{itemize}
 \item[\textup{\textbf{I}.}] If $0<\lambda<1$ and $\eta^2\leq\lambda(1-\lambda)$, then
\begin{equation*}
\theta^{(\eta,\lambda)}_{n,k+1}-\theta^{(\eta,\lambda)}_{n,k}\leq \left(\frac{\pi^2}{f_{n}^{(\eta,\lambda)}}\right)^{1/2}, \quad k=1,\ldots,n-1.
\end{equation*}
	\item[\textup{\textbf{II}.}] If $\lambda>1$ and $\eta\in\R$, then
\begin{equation*}
\theta^{(\eta,\lambda)}_{n,k+1}-\theta^{(\eta,\lambda)}_{n,k}\geq \left(\frac{\pi^2}{f_{n}^{(\eta,\lambda)}}\right)^{1/2}, \quad k=1,\ldots,n-1.
\end{equation*}
	\item[\textup{\textbf{III}.}] If $\lambda=1$, $\eta\in\R$ and $n\geq 4$, then
	\begin{equation*}
\theta^{(\eta,\lambda)}_{n,k+1}-\theta^{(\eta,\lambda)}_{n,k}> \left(\frac{\pi^2}{g_{n}^{(\eta)}}\right)^{1/2}, \quad k=1,\ldots,n-1.
\end{equation*}
\end{itemize}
\end{corollary}
\begin{proof} Notice that $\Delta_{n}(\theta)$, given in \eqref{RomonoviskyEDO S-L form solution with theta}, has a single critical point for $\lambda\neq 1$ at $\theta=\textup{arctan}(\mathfrak{M}_{n})$, where $\mathfrak{M}_{n}$ was defined in Theorem \ref{Thm-convexityR2Zeros}. Under the condition \textup{\textbf{I}}, we have
$\Delta_{n}(\theta)\geq f_{n}^{(\eta,\lambda)}$, and since $y(\theta)=\sin[(f_{n}^{(\eta,\lambda)})^{1/2}\,(\theta-\theta_{n,k}^{(\eta,\lambda)})]$ is the solution of the differential equation $y''(\theta)+f_{n}^{(\eta,\lambda)}y(\theta)=0$, the result is a consequence of Theorem  \ref{SturmTheoClassic}. The condition \textup{\textbf{II}} follows analogously.

In view of \textup{\textbf{III}}, we have $\Delta_{n}(\theta)=2\eta(n+1)\tan(\theta)+(n+1)^{2}-\eta^{2}$. By Theorem \ref{Thm-bounds} the zeros of $\PP_{n}(b;x)$, for $n\geq 4$, are in the interval $(x_{1},x_{2})$ where
\begin{equation*}
 x_{1,2}=\frac{\eta(n^{2}+2)\pm(n-1)\sqrt{\eta^{2}(n^{2}+2n+4)+3n^{2}}}{3n}.
\end{equation*}
Therefore,  $\textup{arctan}(x_{1})\leq\theta\leq\textup{arctan}(x_{2})$ implies that $\Delta_{n}(\theta)<2\eta(n+1)x_{2}+(n+1)^{2}-\eta^{2}$ for $\eta\geq 0$ and
 $\Delta_{n}(\theta)<2\eta(n+1)x_{1}+(n+1)^{2}-\eta^{2}$ for $\eta<0$, and thus $\Delta_{n}(\theta)<g_{n}^{(\eta)}$.
\end{proof}

The next theorem was inspired by the construction in Moak et al. \cite{Saff} used to obtain information about the density of the zeros of Jacobi polynomials.

\begin{theorem}\label{densityR2Zeros}
%Let $-\infty < t_1 < t_2 < \infty$.
Let $\mathcal{X}^{(b)}$ be the set of all the zeros of $\PP_{n}(b;x)$ for every $n \geq 1$. If $\lambda \geq 1$, then $\mathcal{X}^{(b)}$ is dense in $\mathbb{R}$.
\end{theorem}
\begin{proof}

Let $\Theta_{t_1}$ and $\Theta_{t_2}$ be arbitrary numbers in $(-\pi/2, \pi/2)$ such that
$$
\tan(\Theta_{t_1}) = t_{1}, \ \ \tan(\Theta_{t_2}) = t_2 \ \ \mbox{and} \ \ \Theta_M = \max{(|\Theta_{t_1}|,|\Theta_{t_2}|)}.
$$
We first show that for a given $\epsilon > 0$, there exists $n _1$ such that all consecutive zeros $\theta_{n_1,k}^{(\eta,\lambda)}$ and $\theta_{n_1,k+1}^{(\eta,\lambda)}$ of $v_{n_1}$  within $(\Theta_{t_1},\Theta_{t_2})$ are such that $\theta_{n_1,k+1}^{(\eta,\lambda)} -\theta_{n_1,k}^{(\eta,\lambda)} < \epsilon$.
If  $\lambda>1$ and $(n^{2}+2n\lambda+\lambda-n-\eta^{2}) > 0$, consider the two roots $\Theta_{1}^{(n)}$ and $ \Theta_2^{(n)}$ of $\Lambda_{n}-n$ given by
\[
\tan(\Theta_{1,2}^{(n)}) =\frac{\eta(n+\lambda)\pm\sqrt{\eta^{2}(n+\lambda)^{2}+\lambda(\lambda-1)(n^{2}+2n\lambda+  \lambda-n-\eta^{2})}}{\lambda(\lambda-1)},
\]
where $\tan(\Theta_{2}^{(n)}) > \tan(\Theta_1^{(n)})$ and $\Lambda_{n}(\theta)$ is defined in \eqref{RomonoviskyEDO S-L form solution with theta}. Since $\lambda>1$, we have $\Lambda_{n}(\theta)-n>0$ for $\theta\in(\Theta_{1}^{(n)},\Theta_{2}^{(n)})$. Without much difficulty, we can also verify that $\Theta_{1}^{(n)}$ tends to $-\pi/2$ and $\Theta_{2}^{(n)}$ tends to $\pi/2$ when $n$ tends to $\infty$. Let $n_0$ be such that for all $n \geq n_0$, $\Theta_1^{(n)} < \Theta_{t_1}$  and $\Theta_2^{(n)} > \Theta_{t_2}$.
Now we note that the roots of $\sin(\sqrt{n}\theta)$, which is the solution of the differential equation $Y^{\prime\prime}+nY=0$, are dense in the real line as $n$ tends to $\infty$. Therefore, for a given $\epsilon > 0$, let $n_1$ be such that $n_1 > n_0$ and $2\pi/ \sqrt{n_1} < \epsilon$.
By Theorem \ref{SturmTheoClassic} it follows that $v_{n_1}(\theta)$ has a zero between any two consecutive zeros of $\sin(\sqrt{n_1}\theta)$ in $(\Theta_{t_1} , \Theta_{t_2} )$. Thus, we conclude that there are infinitely many zeros of $v_{n}(\theta)$ in $(\Theta_{t_1} , \Theta_{t_2} )$ as $n$ tends to $\infty$, and that if $\theta_{n_1,k}^{(\eta,\lambda)}$ and $\theta_{n_1,k+1}^{(\eta,\lambda)}$ are in $(\Theta_{t_1} , \Theta_{t_2} )$, then $\theta_{n_1,k+1}^{(\eta,\lambda)}- \theta_{n_1,k}^{(\eta,\lambda)} < \epsilon$. Hence, if $x_{n_1,k+1}^{(\eta,\lambda)}$ and $ x_{n_1,k}^{(\eta,\lambda)} $ are in $(t_1, t_2)$, then from $\tan(A-B) = [\tan(A) - \tan(B)]/[1+ \tan(A)\tan(B)]$,
\begin{align*}
x_{n_1,k+1}^{(\eta,\lambda)}- x_{n_1,k}^{(\eta,\lambda)}  &= \left[1+\tan(\theta_{n_1,k+1}^{(\eta,\lambda)})\tan(\theta_{n_1,k}^{(\eta,\lambda)})\right]\tan(\theta_{n_1,k+1}^{(\eta,\lambda)}-\theta_{n_1,k}^{(\eta,\lambda)})\\
&< \left[1+\tan^2(\Theta_M)\right] \tan(\epsilon).
\end{align*}

Now let $\lambda=1$. Then the equation $\Lambda_{n}(\theta)-n=0$ has the root
\[
\tan(\Theta_{1}^{(n)})=-\frac{n^{2}+n+1-\eta^{2}}{2\eta(n+1)}.
\]
Note that when $\eta > 0$, it follows that $\Lambda_{n}(\theta)-n > 0$ for $\theta \in(\Theta_{1}^{(n)}, \pi/2)$ and that $\Theta_{1}^{(n)}$  tends to $-\pi/2$ as $n$ tends to $\infty$. By using the same argument as before, we conclude that $\mathcal{X}^{(b)}$ is dense in $\mathbb{R}$ for $\eta>0$.
The case $\eta<0$ follows from the identity $(-1)^{n}\PP_{n}(\overline{b};-x)=\PP_{n}(b;x)$, with $b=\lambda+i\eta$. For $\eta=0$, the solution of Equation  \eqref{RomonoviskyEDO S-L form with theta} is a linear combination of $\cos((n+1)\theta)$ and $\sin((n+1)\theta)$, whose zeros are dense in $\mathbb{R}$ as $n$ tends to $\infty$.
This concludes the proof of the theorem.
\end{proof}

We now consider the polynomials $\{\PP_{n}^{(k)}(b;x)\}_{n \geq 0}$, the $k$-th associated CRR polynomials,  obtained from
\begin{equation}\label{3termR2-k}
\PP_{n+1}^{(k)}(b;x)=(x-c_{k+n+1}^{(b)})\PP_{n}^{(k)}(b;x)-d_{k+n+1}^{(\lambda)}(x^{2}+1)\PP_{n-1}^{(k)} (b;x),  \quad  n \geq 1,
\end{equation}
where $\PP_{0}^{(k)}(b;x)=1$,  $\PP_{1}^{(k)}(b;x)=x-c_{k+1}^{(b)}$. The coefficients $c_{n}^{(b)}$ and $d_{n+1}^{(\lambda)}$ are as in  \eqref{RomonovskyCoeff}. Clearly, $\PP_{n}^{(0)}(b;x) = \PP_{n}(b;x)$, $n \geq 1$.

Then Theorem \ref{densityR2Zeros} has the corollary

\begin{corollary}\label{Cor-densityR2Zeros}
%Let $-\infty < t_1 < t_2 < \infty$.
For any fixed nonnegative integer $k$, let $\mathcal{X}^{(b)}_{k}$ be the set of all the zeros of $\PP_{n}^{(k)}(b;x)$, for every $n \geq 1$. If $\lambda \geq 1$, then $\mathcal{X}^{(b)}_{k}$ is dense in $\mathbb{R}$.
\end{corollary}

\begin{proof}
Consider the positive chain sequence $\{d_{n}^{(\lambda)}\}_{n \geq 2}$ (see  \eqref{Eq-MinMaxPSeq-1} and \eqref{Eq-MinMaxPSeq-2} below for more details). It follows that for any $k \geq 1$ the sequence $\{d_{k+n}^{(\lambda)}\}_{n \geq 2}$  is also a positive chain sequence (details in Chihara \cite{Chihara-book1978}). Consequently,  in the same way as shown in Ismail and Ranga \cite{Ismail-Ranga}, it follows that the zeros of $\PP_{n}^{(k)}(b;x)$ are all real, simple and also that  the zeros of $\PP_{n-1}^{(k)}(b;x)$ and $\PP_{n}^{(k)}(b;x)$ interlace.

From the three term recurrence in Equation \eqref{3termR2-k} it is also straightforward to show that
\[
\PP_{n}^{(k+1)}(b;x)\PP_{n}^{(k)}(b;x) - \PP_{n-1}^{(k+1)}(b;x)\PP_{n+1}^{(k)}(b;x) = (1+x^2)^{n} \prod_{j=2}^{n+1}d_{k+j}^{(\lambda)} \ > \ 0,
\]
for all real $x$ and $n \geq 1$. Thus, the zeros of $\PP_{n}^{(k+1)}(b;x)$ and $\PP_{n+1}^{(k)}(b;x)$ also interlace.

Hence, the corollary follows via induction.
% não
\end{proof}

We now provide some information regarding the orthogonality properties of the sequence of polynomials $\{\PP_{n}^{(k)}(b;x)\}_{n\geq 0}$. These orthogonality properties are very much related to what it is known about the positive chain sequence $\{d_{k+n}^{(\lambda)}\}_{n \geq 2}$, and the Verblunsky coefficients associated with orthogonal polynomials on the unit circle. For more information about Verblunsky coefficients and  orthogonal polynomials on the unit circle we cite, for example, Ismail  \cite{Ismail-Book2005} and Simon \cite{Simon-Book-p1}.

The sequence $\{d_{n}^{(\lambda)}\}_{n \geq 2}$ is one of the nicest example of a positive chain sequence, where one can state its minimal $\{m_{n}^{(\lambda)}\}_{n \geq 1}$ and maximal $\{M_{n}^{(\lambda)}\}_{n \geq 1} = \{M_{n}^{(\lambda;0)}\}_{n \geq 1}$ parameter sequences explicitly as follows:
\begin{equation} \label{Eq-MinMaxPSeq-1}
  \begin{array}{l}
     \dsp m_{n}^{(\lambda)} = M_{n}^{(\lambda;0)} = \frac{1}{2}\frac{n-1}{\lambda+n-1}, \quad n \geq 1,
  \end{array}
\end{equation}
if $0 < \lambda \leq 1/2$ and
\begin{equation} \label{Eq-MinMaxPSeq-2}
  \begin{array}{l}
     \dsp m_{n}^{(\lambda)} = \frac{1}{2}\frac{n-1}{\lambda+n-1} \quad \mbox{and} \quad  M_{n}^{(\lambda;0)} = \frac{1}{2}\frac{2\lambda+n-2}{\lambda+n-1},  \quad n \geq 1,
  \end{array}
\end{equation}
if $\lambda > 1/2$.

Hence, for any $k \geq 1$, there follows (see Chihara \cite{Chihara-book1978}, page 94)  that the  maximal $\{M_{n}^{(\lambda;k)}\}_{n \geq 1}$ parameter sequences of the positive chain sequence $\{d_{k+n}^{(\lambda)}\}_{n \geq 2}$ satisfy
\[
  \begin{array}{l}
    \dsp  M_{n}^{(\lambda;k)}  = \frac{1}{2}\frac{k+n-1}{\lambda+k+n-1}, \quad n \geq 1, \quad \mbox{if} \quad 0 < \lambda \leq 1/2,  \\[3ex]
    \dsp  M_{n}^{(\lambda;k)}  = \frac{1}{2}\frac{2\lambda+k+n-2}{\lambda+k+n-1}, \quad n \geq 1, \quad \mbox{if} \quad \lambda > 1/2.
  \end{array}
\]
Thus, we can state the following:

\begin{theorem}
   Assume that either  $k = 0$ and $\lambda > 1/2$ or  $k \geq 1$ and $\lambda > 0$. Let $\nu^{(b;k)}$ be the probability measure on the unit circle such that its Verblunsky coefficients are
\[
   \beta_{n-1}^{(b;k)} = \frac{1}{\tau_{n-1}^{(b;k)}} \frac{1 - M_{n}^{(\lambda;k)} - i c_{k+n}^{(b)}}{1-ic_{k+n}^{(b)}}, \quad n \geq 1,	
\]
where $\tau_{0}^{(b;k)} = 1$ and $\tau_{n}^{(b;k)} = \tau_{n-1}^{(b;k)} (1 - i c_{k+n}^{(b)})/(1 + i c_{k+n}^{(b)})$.

If $d\varphi^{(b;k)}(x) = d\nu^{(b;k)}((x+i)/(x-i))$ then
\[
    \int_{-\infty}^{\infty} x^{j} \frac{\PP_{n}^{(k)}(b;x)}{(x^2+1)^n} d\varphi^{(b;k)}(x) = \gamma_{n}^{(b;k)}\, \delta_{n,j}, \quad j = 0,1, \ldots, n,
\]
where $\gamma_{0}^{(b;k)} = 1$ and $\gamma_{n}^{(b;k)} = (1-M_{n}^{(\lambda;k)}) \gamma_{n-1}^{(b;k)}$,  $n \geq 1$. Moreover,
\begin{equation}\label{associated}
    \PP_{n}^{(k+1)}(b;x) = \int_{-\infty}^{\infty}\frac{(1+x^{2})^{n}\PP_{n}^{(k)}(t)-(1+t^{2})^{n}\PP_{n}^{(k)}(x)}{M_{1}^{(\lambda;k)}(t-x)(1+t^{2})^{n}}d\varphi^{(b;k)}(t).
\end{equation}
\end{theorem}

\begin{proof}
The orthogonality relation follows from Theorem 1.2 of Ismail and Ranga \cite{Ismail-Ranga}. To obtain Formula \eqref{associated} we use results found in  Section 3 of Bracciali et al. \cite{BraccPereiRanga-NA2019}.
\end{proof}

In general, it is not possible to write down the exact expression for $\varphi^{(b;k)}$, with the exception
$$
    d\varphi^{(b;0)}(x) = \frac{2^{2\lambda-1} |\Gamma(b)|^2}{\Gamma(2\lambda-1)}   \frac{e^{\eta \pi}}{2\pi} \frac{(e^{-\cot^{-1} x})^{2\eta}}{(1+x^2)^{\lambda}} dx.
$$

\section{Asymptotics}  \label{main2}
\setcounter{equation}{0} \setcounter{remark}{0} \setcounter{theorem}{0}\setcounter{corollary}{0}\setcounter{lemma}{0}

Continuing our proposed program, we now study the asymptotics of CRR polynomial. Recall that the these polynomials are parametrized by $b=\lambda+i\eta$. We shall focus here on the asymptotic properties of $\PP_{n}(b;x)$ with respect to both parameters and their link to the Hermite and Laguerre polynomials. In order to show it, we first recall that the polynomials $H_{n}(x)$ and $L_n^{(\alpha)}(x)$ are solutions of the recurrence formulas
\[
   H_{n+1}(x) = 2x H_{n}(x) - 2n H_{n-1}(x), \quad n \geq 1,
\]
with $H_{0}(x) = 1$, $H_{1}(x) = 2x$, and
\[
  L_{n+1}^{(\alpha)}(x) = \left(2 + \frac{\alpha-1-x}{n+1}\right)L_{n}^{(\alpha)}(x)  - \left(1 + \frac{\alpha-1}{n+1}\right)L_{n-1}^{(\alpha)}(x), \quad n \geq 1,
\]
with $L_{0}^{(\alpha)}(x) = 1$, $L_{1}^{(\alpha)}(x) = \alpha+1-x$, respectively.  We also recall that $\{H_n (x)\}_{n \geq 0}$ is a sequence of orthogonal polynomials in $(-\infty, \infty)$ with respect to the Gaussian weight $e^{-x^2}$ and $\{L_n^{(\alpha)}(x)\}_{n \geq 0}$ is orthogonal in $(0, \infty)$ with respect to the weight function $x^{\alpha} e^{-x}$.

\begin{theorem}\label{Assintotico-eta-lambda} For each $n\geq 0$, let
\begin{equation*}
   \widehat{U}_n(b;x)=\frac{2^{n}(\lambda)_{n}}{\lambda^{n/2}} \PP_{n}(b;x/\sqrt{\lambda}) \quad  \mbox{and} \quad  \widetilde{U}_{n}(b;x)
   =\frac{(\lambda)_{n}}{\eta^n n!}x^{n}\PP_{n}(b;2\eta/x).
\end{equation*}
Then the following expansion formulas hold:
\begin{align*}
  \begin{split}
 \widehat{U}_n(b;x)& = H_{n}(x) - 2\eta n H_{n-1}(x) \frac{1}{\sqrt{\lambda}} + \mathcal{O}\big((1/\sqrt{\lambda})^2\big), \\
\widetilde{U}_n(b;x)&=L_{n}^{(2\lambda-1)}(x)+\mathcal{O}\big(1/\eta^2\big),
  \end{split}
\end{align*}
where by $\mathcal{O}\left(h^{k}\right)$ we mean powers of $h$ greater than or equal to $k$.

\end{theorem}

\begin{proof}
From Equation \eqref{3termR2TTT}, we first observe that the sequence of functions $\{\widehat{U}_n(b;x)\}_{n \geq 0}$ satisfies
\begin{equation}\label{recurrenceQ}
  \begin{array}{l}
    \widehat{U}_0(b;x)=H_{0}(x)=1, \quad \widehat{U}_1(b;x)=2x-\frac{2\eta}{\sqrt{\lambda}}\quad \mbox{and} \\[2ex]
\widehat{U}_{n+1}(b;x)=2[x-\alpha_{n+1}(x)]\widehat{U}_{n}(b;x)-[2n+\beta_{n+1}(x)]\widehat{U}_{n-1}(b;x),
  \end{array}
\end{equation}
for $n \geq 1$, where
\begin{equation*}
\alpha_{n+1}(x)=\frac{\eta}{\sqrt{\lambda}}-\frac{n}{\lambda}x \ \ \mbox{and} \ \ \beta_{n+1}(x)=\frac{n[2x^2+n-1]}{\lambda}+\frac{n(n-1)x^2}{\lambda^2}.
\end{equation*}
Hence, $\widehat{U}_1(b;x) =H_{1}(x)-\frac{2\eta}{\sqrt{\lambda}} H_{0}(x) $ and the claim follows trivially for $\widehat{U}_1(b;x)$. Suppose that $$\widehat{U}_k(b;x) = H_{k}(x) - \frac{2\eta k}{\sqrt{\lambda}} H_{k-1}(x)  + \mathcal{O}\big((1/\sqrt{\lambda})^2\big)$$ for each $k=1,\ldots, n$. By the Recurrence \eqref{recurrenceQ}, we have
\begin{equation*}
\begin{array}{ll}
\widehat{U}_{n+1}(b;x) = 2\left(x- \frac{\eta}{\sqrt{\lambda}}\right)\widehat{U}_{n}(b;x)-2n\widehat{U}_{n-1}(b;x)+\mathcal{O}\big((1/\sqrt{\lambda})^2\big)\\[1ex]
\hspace{10ex} = 2\left(x- \frac{\eta}{\sqrt{\lambda}}\right)[H_{n}(x) - \frac{2\eta n}{\sqrt{\lambda}} H_{n-1}(x) ]   - 2n[H_{n-1}(x) - \frac{2\eta (n-1)}{\sqrt{\lambda}} H_{n-2}(x) ] + \mathcal{O}\big((1/\sqrt{\lambda})^2\big)\\[1ex]
\hspace{10ex} =2xH_{n}(x)-2nH_{n-1}(x)  - \frac{2\eta}{\sqrt{\lambda}}[H_{n}(x) +2nx H_{n-1}(x) - 2n(n-1) H_{n-2}(x)]  +    \mathcal{O}\big((1/\sqrt{\lambda})^2\big).\
\end{array}
\end{equation*}
Thus,
\begin{equation*}
  \begin{array}{l}
\widehat{U}_{n+1}(b;x) = H_{n+1}(x) - \frac{2\eta(n+1)}{\sqrt{\lambda}} H_{n}(x) + \mathcal{O}\big((1/\sqrt{\lambda})^2\big),
  \end{array}
\end{equation*}
and this proves the expansion formula for $\widehat{U}_{n}(b;x)$, $n \geq 2$.

Now to obtain the second expansion formula, we first observe that
%
%\begin{equation*}\label{recurrenceU-eta}
%\begin{array}{l}
\begin{equation*}
\widetilde{U}_0(b;x)= 1 = L_{0}^{(2\lambda-1)}(x), \quad \widetilde{U}_{1}(b;x) =2\lambda - x = L_{1}^{(2\lambda-1)}(x),
\end{equation*}
and
%\begin{align*}
%\displaystyle \widetilde{U}_{n+1}(b;x)=&\left[2 + \frac{2\lambda-2-x}{n+1}\right]\widetilde{U}_{n}(b;x)-\\
%&\left[1 + \frac{2\lambda-2}{n+1}\right]\left[\left(\frac{x}{2\eta}\right)^2+1\right]\widetilde{U}_{n-1}(b;x),
%\end{align*}
\begin{align*}
 \widetilde{U}_{n+1}(b;x)=\Big[2 + \frac{2\lambda-2-x}{n+1}\Big]\widetilde{U}_{n}(b;x)-\Big[1 + \frac{2\lambda-2}{n+1}\Big]\Big[1+\Big(\frac{x}{2\eta}\Big)^2\Big]\widetilde{U}_{n-1}(b;x),
\end{align*}
%\end{array}
%\end{equation*}
%
for $n \geq 1$.  Thus, the formula is obtained similarly.
\end{proof}

We note that the asymptotic behavior with respect to the parameter $\eta$ can also be deduced from the connection between the Romanovski-Routh polynomials and the Bessel polynomials, as shown in Lesky et al. \cite{Book-HOPqA-2010}. The next theorem presents the order of convergence for the zeros of CRR polynomial as functions of $\lambda$ and $\eta$.

\begin{theorem}\label{monotonicity_02}
Let $h_{n,k}$ and $l^{(2\lambda-1)}_{n,k}$ be the $k$-th zeros (in increasing order) of $H_n (x)$ and $L_{n}^{(2\lambda-1)}(x)$, respectively. Then
\begin{equation*}
x_{n,k}^{(\eta,\lambda)} = \frac{h_{n,k}}{\sqrt{\lambda}} + \frac{\eta}{\lambda} + \mathcal{O}\big((1/\sqrt{\lambda})^{3}\big) \ \
\mbox{and} \ \ \
x_{n,k}^{(\eta,\lambda)} = \frac{2\eta}{l^{(2\lambda-1)}_{n,n+1-k}} +\mathcal{O}(1/\eta),
\end{equation*}
for each $k=1,\ldots,n$, as $\lambda\rightarrow\infty$ and $\eta\rightarrow\infty$, respectively. The symbol $\mathcal{O}$ is as in Theorem  \ref{Assintotico-eta-lambda}
\end{theorem}

\begin{proof}
The asymptotic behavior given by Theorem \ref{Assintotico-eta-lambda} suggests that the $\lambda$ dependence of $x_{n,k}^{(\eta,\lambda)}$ can be modeled by the ansatz
\begin{equation}\label{aux_01}
x_{n,k}^{(\eta,\lambda)}=\frac{h_{n,k}}{\sqrt{\lambda}}+\frac{B_{n,k}}{\lambda} + \mathcal{O}\big((1/\sqrt{\lambda})^{3}\big).
\end{equation}
Then our claim holds true if we show that $B_{n,k} = \eta$\, for $k =1,2, \ldots, n$.

By straightforward calculations we can show that
\[
      \frac{\PP_{n}^{\prime\prime}(b;x_{n,k}^{(\eta,\lambda)}) }{2\PP_{n}^{\prime}(b;x_{n,k}^{(\eta,\lambda)}) } = \sum_{i=1 \atop i\ne k}^{n}\frac{1}{x_{n,k}^{(\eta,\lambda)}-x_{n,i}^{(\eta,\lambda)}} .
\]
Thus, from Equation \eqref{Eq-Dif-Eq-2} we have
\begin{equation}\label{aux_02}
   \sum_{i=1 \atop i\ne k}^{n}\frac{1}{x_{n,k}^{(\eta,\lambda)}-x_{n,i}^{(\eta,\lambda)}}
              =\frac{(n-1+\lambda)x_{n,k}^{(\eta,\lambda)}-\eta}{1+(x_{n,k}^{(\eta,\lambda)})^{2}}, \quad \mbox{for} \quad k=1,2, \ldots n.
\end{equation}
Moreover, by using Equation \eqref{aux_01} in the left hand side of   \eqref{aux_02} and then considering the formal series expansion in powers of $1/\sqrt{\lambda}$, we have
\begin{align*}
\begin{split}
\sum_{i=1 \atop i\ne k}^{n}\frac{1}{x_{n,k}^{(\eta,\lambda)}-x_{n,i}^{(\eta,\lambda)}}
     &=  \sum_{i=1 \atop i\ne k}^{n}\frac{\sqrt{\lambda}}{h_{n,k}-h_{n,i}+(B_{n,k}-B_{n,i})/\sqrt{\lambda} + \mathcal{O}\big((1/\sqrt{\lambda})^2\big)}\\
&=\sum_{i=1 \atop i\ne k}^{n}\frac{\sqrt{\lambda}}{h_{n,k}-h_{n,i}}-
\sum_{i=1 \atop i\ne k}^{n}\frac{B_{n,k}-B_{n,i}}{(h_{n,k}-h_{n,i})^{2}} + \mathcal{O}\big(1/\sqrt{\lambda}\big),
\end{split}
\end{align*}
for $k =1,2, \ldots,n$. Similarly, substitution of Equation \eqref{aux_01} in the right hand side of \eqref{aux_02} gives
\begin{align*}
\begin{split}
\frac{(n-1+\lambda)x_{n,k}^{(\eta,\lambda)}-\eta}{1+(x_{n,k}^{(\eta,\lambda)})^{2}} &
    =  \frac{(n-1+\lambda) \left[h_{n,k}/\sqrt{\lambda} +B_{n,k}/\lambda +\mathcal{O}\big((1/\sqrt{\lambda})^3\big)\right]-\eta}
{1+h_{n,k}^{2}/\lambda+\mathcal{O}\big((1/\sqrt{\lambda})^3\big)}\\
&=\sqrt{\lambda}h_{n,k}+B_{n,k}-\eta+\mathcal{O}\big(1/\sqrt{\lambda}\big),
 \end{split}
\end{align*}
for $k =1,2, \ldots,n$. Thus, by comparing the above equalities, we find that the constraints
\begin{equation}\label{System-for-lambda}
  \sum_{i=1 \atop i\ne k}^{n}\frac{1}{h_{n,k}-h_{n,i}} = h_{n,k} \quad \mbox{and} \quad
   -\sum_{i=1 \atop i\ne k}^{n}\frac{B_{n,k}-B_{n,i}}{(h_{n,k}-h_{n,i})^{2}}=B_{n,k}-\eta,
\end{equation}
should be satisfied for each $k=1,\ldots,n$.
The first one involving the zeros of Hermite polynomials is known to be true. It is commonly known as Stieltjes system of equations for the zeros of the Hermite polynomials, studied by Stieltjes in 1885 (see, for example, Stieltjes \cite{Stieltjes1885-1} and Szeg\H{o} \cite{szego}).

The second constraint in \eqref{System-for-lambda} is a $n \times n$ linear system of equations that can be put in the form $\mathbf{T}\textbf{v}=-\textbf{v}$, where the matrix $\mathbf{T} = (t_{k,j})_{k,j=1}^{n}$ is symmetric with diagonal and off-diagonal elements given, respectively, by
\begin{equation*}
t_{k,k}=\sum_{i=1 \atop i\ne k}^{n}\frac{1}{(h_{n,k}-h_{n,i})^{2}} \quad \textup{and} \quad t_{k,j}= -\frac{1}{(h_{n,j}-h_{n,k})^{2}}.
\end{equation*}

Moreover,  $\mathbf{v} = [v_{1}, v_{2}, \ldots, v_{n}]^t$ with $v_{k}=B_{n,k}-\eta$.
The matrix $\mathbf{T}$ is positive semi-definite (for instance, in Calogero \cite{Calogero-LNC1977} its eigenvalues are explictly given) and thus its eigenvalues are nonnegative. This implies that the unique solution of the system $\mathbf{T}\mathbf{v}=-\mathbf{v}$ is the trivial solution. That is, $v_{k}=B_{n,k}-\eta = 0$ for $k=1,\ldots,n$. This proves the asymptotics with respect to the parameter $\lambda$.

To establish the asymptotic behavior with respect to the parameter $\eta$, suppose that
\begin{equation}\label{aux_03}
    x_{n,k}^{(\eta,\lambda)}=2\eta D_{n,k}+C_{n,k}+\mathcal{O}\left(1/\eta\right).
\end{equation}
From Theorem \ref{Assintotico-eta-lambda} we must have  $D_{n,k}=1/l_{n,n+1-k}^{(2\lambda-1)}$. Substitution of Equation \eqref{aux_03} in Equation  \eqref{aux_02}, followed  by the expansion in powers of $1/\eta$ leads to the constraints
\begin{align}\label{ttttt}
  \begin{split}
    \sum_{i=1 \atop i\ne k}^{n}\frac{1}{D_{n,k}-D_{n,i}}&=\frac{2(n+\lambda-1)D_{n,k}-1}{2 D_{n,k}^{2}}, \\
     \sum_{i=1 \atop i\ne k}^{n}\frac{C_{n,k}-C_{n,i}}{(D_{n,k}-D_{n,i})^2} &= \frac{C_{n,k}[(n+\lambda-1)D_{n,k}-1]}{D_{n,k}^{3}},
  \end{split}
\end{align}
for $k=1,2, \ldots,n$.  The first expression involving the reciprocal zeros of Laguerre polynomials is the nonlinear system satisfied by the zeros of the special Bessel polynomial, namely, $x^nL^{(2\lambda-1)}_{n}(1/x)$. 
%For more details about such nonlinear systems for the zeros of classical polynomials we cite Alici and Ta\c{s}eli \cite{Alici-Tacaseli-2015}.

In terms of $l_{n,n+1-k}^{(2\lambda-1)}$, Equations \eqref{ttttt} are equivalent to
\begin{align} \label{Eq-111}
  \begin{split}
     \sum_{i=1 \atop i\ne k}^{n}\frac{l_{n,n+1-i}^{(2\lambda-1)}}{l_{n,n+1-i}^{(2\lambda-1)}-l_{n,n+1-k}^{(2\lambda-1)}} &= (n+\lambda-1) - \frac{1}{2} l_{n,n+1-k}^{(2\lambda-1)}, \\
     \sum_{i=1 \atop i\ne k}^{n}\frac{[l_{n,n+1-i}^{(2\lambda-1)}]^2(C_{n,k}-C_{n,i})}{\left(l_{n,n+1-i}^{(2\lambda-1)}-l_{n,n+1-k}^{(2\lambda-1)}\right)^2} &= C_{n,k}\left[(n+\lambda-1)- l_{n,n+1-k}^{(2\lambda-1)}\right],
  \end{split}
\end{align}
for $k=1,2, \ldots,n$.

Now substituting for $(n+\lambda-1)-l_{n,n+1-k}^{(2\lambda-1)}/2$ in the right hand side of the second equation in \eqref{Eq-111} with the first equation, we find the system
\begin{equation*}
\begin{array}{rl}
\dsp \sum_{i=1 \atop i\ne k}^{n}
\left[
\frac{l_{n,n+1-i}^{(2\lambda-1)}l_{n,n+1-k}^{(2\lambda-1)}C_{n,k}}{\left(l_{n,n+1-i}^{(2\lambda-1)}-l_{n,n+1-k}^{(2\lambda-1)}\right)^2}
 -
\frac{l_{n,n+1-i}^{(2\lambda-1)}l_{n,n+1-i}^{(2\lambda-1)}C_{n,i}}{\left(l_{n,n+1-i}^{(2\lambda-1)}- l_{n,n+1-k}^{(2\lambda-1)}\right)^2}
\right]
=
-\frac{1}{2}l_{n,n+1-k}^{(2\lambda-1)}C_{n,k}, \quad \mbox{for} \quad k=1,2, \ldots,n.
\end{array}
\end{equation*}
We now consider this as a system of equations written as 
$$
\tilde{\mathbf{T}}\mathbf{u}=-\frac{1}{2}\mathbf{u}, 
\ \ \mathbf{u}= [ l_{n,n}^{(2\lambda-1)}C_{n,1},\, l_{n,n-1}^{(2\lambda-1)}C_{n,2}, \ldots,\, l_{n,1}^{(2\lambda-1)}C_{n,n}]^{T}.
$$
%, where $\mathbf{u}= [ l_{n,1}^{(2\lambda-1)}C_{n,1},\, l_{n,2}^{(2\lambda-1)}C_{n,2}, \ldots,\, l_{n,n}^{(2\lambda-1)}C_{n,n}]^{T}$. 
The matrix $\tilde{\mathbf{T}}$ is positive semi-definite, which follows from the Gershgorin Theorem. This system was also considered in Calogero \cite{Calogero-LNC1977} where its eigenvalues were calculated. Thus, with the negative factor $-1/2$, the system  $\tilde{\mathbf{T}}\mathbf{u}=-\frac{1}{2}\mathbf{u}$ has only the trivial solution. That is, $l_{n,n+1-k}^{(2\lambda-1)}C_{n,k} = 0$, $k =1, 2,\ldots, n$ and hence, $C_{n,k} = 0$, $k =1,2, \ldots, n$.
\end{proof}

%\bibliographystyle{unsrtnat}
%\bibliography{references}  %%% Uncomment this line and comment out the ``thebibliography'' section below to use the external .bib file (using bibtex) .


\begin{thebibliography}{1}

\bibitem{MRRT-CRR-2019}
Mart\'inez-Finkelshtein A, Silva~Ribeiro LL, Sri~Ranga A, Tyaglov M.
  Complementary Romanovski-Routh polynomials: from orthogonal polynomials on
  the unit circle to Coulomb wave functions.  {\it Proc. Amer. Math. Soc.
  }2019;147:2625--2640.
\newblock DOI:10.1090/proc/14423.

\bibitem{RapWebCastKirc-CEJP2007}
Raposo AP, Weber HJ, Alvarez-Castillo DE, Kirchbach M. Romanovski polynomials
  in selected physics problems.  {\it Centr. Eur. J. Phys. }2007;5:253--284.
\newblock DOI:10.2478/s11534-007-0018-5.

\bibitem{Gautschi-SIAMRev1967}
Gautschi W. Computational aspects of three-term recurrence relations.  {\it
  SIAM Rev. }1967;9:24--82.
\newblock DOI:10.1137/1009002.

\bibitem{Swaminathan2023}
Shukla V, Swaminathan A. Spectral properties related to generalized complementary Romanovski–Routh polynomials. 
{\it Rev. Real Acad. Cienc. Exactas Fis. Nat. Ser. A-Mat. } 203;117, 78.
\newblock DOI:10.1007/s13398-023-01410-0.

\bibitem{BraccPereiRanga-NA2019}
Bracciali CF, Pereira JA, Sri~Ranga A. Quadrature rules from a $R_{II}$ type
  recurrence relation and associated quadrature rules on the unit circle. 
   {\it  Numer. Algor.} 2020;83:1029--1061.
\newblock DOI:10.1007/s11075-019-00714-w.

\bibitem{Kerstin2014}
Jordaan K, To\'okos F. Orthogonality and asymptotics of Pseudo-Jacobi
  polynomials for non-classical parameters.  {\it J. Approx. Theory.
  }2014;178:1--12.
\newblock DOI:10.1016/j.jat.2013.10.003.


\bibitem{Veronese}
Mart\'inez-Finkelshtein A, Sri~Ranga A, Veronese DO. Extreme zeros in a
  sequence of para-orthogonal polynomials and bounds for the support of the
  measure.  {\it Math. Comp. }2018;87:261--288.
\newblock DOI:10.1090/mcom/3210.

%\bibitem{Lesky-1996}
%Lesky PA. Endliche und unendliche Systeme von kontinuierlichen klassischen
%  Orthogonalpolynomen.  {\it Z. Angew. Math. Mech. }1996;76:181--184.
%\newblock DOI:10.1002/zamm.19960760317.

\bibitem{Ismail-Ranga}
Ismail MHE, Sri~Ranga A. $R_{II}$ type recurrence, generalized eigenvalue
  problem and orthogonal polynomials on the unit circle.  {\it Linear Algebra
  Appl. }2019;562:63--90.
\newblock DOI:10.1016/j.laa.2018.10.005.

\bibitem{Wall-book1948}
Wall HS. {\it Analytic Theory of Continued Fraction}.
\newblock New York: D. Van Nostrand Company; 1948.
\newblock ISBN 9780486823690.

\bibitem{Chihara-book1978}
Chihara TS. {\it An Introdution to Orthogonal Polynomials}.
\newblock New York: Gordon and Breach Science Publishers; 1978.
\newblock ISBN 9780677041506.

\bibitem{Masson}
Ismail MHE, Masson DR. Generalized orthogonality and continued fractions.  {\it
  J. Approx. Theory. }1995;83:1--40.
\newblock DOI:10.1006/jath.1995.1106.

\bibitem{Zhedanov-JAT1999}
Zhedanov A. Biorthogonal rational functions and generalized eigenvalue problem.
   {\it J. Approx. Theory. }1999;101:303--329.
\newblock DOI:10.1006/jath.1999.3339.





\bibitem{szego}
Szeg\H{o} G. {\it Orthogonal Polynomials (American Math. Soc: Colloquium
  publ)}.
\newblock Providence: American Mathematical Society; 1975.
\newblock ISBN 9780821810231.

\bibitem{Hille}
Hille E. {\it Lectures on Ordinary Differential Equation}.
\newblock New York: Addison-Wesley Publishing Company; 1968.
\newblock ISBN 9780201530834.



\bibitem{Kerstin}
Jordaan K, To\'okos F. Convexity of the zeros of some orthogonal polynomials
  and related functions.  {\it J. Comput. Appl. Math. }2009;233:762--767.
\newblock DOI:10.1016/j.cam.2009.02.045.

\bibitem{Geno}
Nikolov G, Uluchev R. Inequalities for Real-Root Polynomials. Proof of a
  Conjecture of Foster and Krasikov.  In: Approximation Theory: A volume
  dedicated to Borislav Bojanov (D.K. Dimitrov, G. Nikolov, and R. Uluchev,
  Eds.), Marin Drinov Academic Publ. House, Sofia:201--216; 2004.

\bibitem{DimitrovGeno}
Dimitrov DK, Nikolov GP. Sharp bounds for the extreme zeros of classical
  orthogonal polynomials.  {\it J. Approx. Theory. }2010;162:1793--1804.
\newblock DOI:10.1016/j.jat.2009.11.006.

\bibitem{Saff}
Moak DS, Saff EB, Varga RS. On the zeros of Jacobi polynomials
  $P_{n}^{(\alpha_{n},\beta_{n})}(x)$.  {\it Trans. Amer. Math. Soc.
  }1979;249:159--162.
\newblock DOI:10.1090/S0002-9947-1979-0526315-8.

\bibitem{Ismail-Book2005}
Ismail MHE. {\it Classical and Quantum Orthogonal Polynomials in one Variable
  (Encyclopedia of Mathematics and its Applications)}.
\newblock Cambridge: Cambridge University Press; 2005.
\newblock ISBN 9781107325982.

\bibitem{Simon-Book-p1}
Simon B. {\it Orthogonal Polynomials on the Unit Circle. Part 1. Classical
  Theory (American Math. Soc: Colloquium publ)}.
\newblock Providence: American Mathematical Society; 2005.
\newblock ISBN 0821837575.



\bibitem{Book-HOPqA-2010}
Lesky PA, Swarttouw RF, Koekoe R. {\it Hypergeometric Orthogonal Polynomials
  and Their q-Analogues (Springer Monographs in Mathematics)}.
\newblock Berlin: Springer-Verlag Berlin Heidelberg; 2010.
\newblock ISBN 9783642050138.

\bibitem{Stieltjes1885-1}
Stieltjes TJ. Sur quelques th\'{e}or\`{e}mes d'alg\`{e}bre.  {\it Comptes
  Rendus de I'Acad\'{e}mie des Sciences. }1885;100:439--440.

\bibitem{Calogero-LNC1977}
Calogero P. On the zeros of the classical polynomials.  {\it Lett. Nuovo
  Cimento. }1977;19:505--508.
\newblock DOI:10.1007/BF02748213.

%\bibitem{Alici-Tacaseli-2015}
%Alici H, Ta\c{s}eli H. Unification of Stieltjes-Calogero type relations for the
%  zeros of classical orthogonal polynomials.  {\it Math. Methods Appl. {S}ci.
%  }2015;38:3118--3129.
%\newblock DOI:10.1002/mma.3285.

\end{thebibliography}




% \end{thebibliography}

\end{document}